\documentclass[12pt]{article}
\usepackage[symbol]{footmisc}
\usepackage[misc]{ifsym}
\usepackage{mathtools}
\usepackage{amsmath}
\usepackage{amssymb}
\usepackage{amsthm}
\usepackage{mathrsfs}
\usepackage{hyperref}
\usepackage{graphicx}
\usepackage{caption}
\usepackage{subcaption}
\usepackage{xcolor}
\usepackage{listings}
 \usepackage[english]{babel}

\usepackage{natbib}

\numberwithin{equation}{section}
\begin{document}

\noindent {STRONG LAW OF LARGE NUMBERS FOR FUNCTIONALS OF RANDOM FIELDS WITH UNBOUNDEDLY INCREASING COVARIANCES}
\vskip 3mm

\vskip 5mm
\noindent Illia Donhauzer\footnote[2]{La Trobe University, Melbourne, Australia. email: I.Donhauzer@latrobe.edu.au}  Andriy Olenko\footnote[3]{ \Letter \ La Trobe University, Melbourne, Australia. email: a.olenko@latrobe.edu.au}  Andrei Volodin\footnote[4]{University of Regina, Saskatchewan, Canada. email: andrei.volodin@uregina.ca}

\vskip 3mm

\noindent ABSTRACT

The paper proves the Strong Law of Large Numbers for integral functionals of random fields with unboundedly increasing covariances. The case of functional data and increasing domain asymptotics is studied.  Conditions to guarantee that the Strong Law of Large Numbers holds true are provided. The considered scenarios include wide classes of  non-stationary random fields. The discussion about application to weak and long-range dependent random fields and numerical examples are given.
\vskip 4mm

\renewcommand{\qedsymbol}{$\blacksquare$}

\theoremstyle{definition}
\newtheorem{definition}{Definition}

\theoremstyle{definition}
\newtheorem{assumption}{Assumption}

\theoremstyle{definition}
\newtheorem{lemma}{Lemma}

\theoremstyle{definition}
\newtheorem{theorem}{Theorem}

\theoremstyle{definition}
\newtheorem{remark}{Remark}

\theoremstyle{definition}
\newtheorem*{example}{Example}

\section{Introduction}

The recent evolution of technology and measuring tools provided a tremendous amount of functional data and substantively motivated a development of new
statistical models and methods. Numerous classical results that originally were obtained for discretely sampled data require extensions to new types of observations collected as functional curves, raster images or spatial data. These high-dimensional data structures often do  not have such nice properties as stationarity or uniform boundedness of their moment characteristics.

The aim of this paper is to derive the Strong Law of Large Numbers (SLLN) for functional data in $\mathbb{R}^d,d\geq 1,$ under not restrictive assumptions on moments and dependencies between observations. We consider realizations of random fields $X(s),\ s\in\mathbb{R}^d,\ d\geq1,$ (not necessary homogeneous and isotropic) with weak restrictions on their covariance functions $B(s_1,s_2),\ s_1,s_2\in \mathbb{R}^d.$ Namely, the covariance functions $B(s_1,s_2)$  can unboundedly increase, as $||s_1||,||s_2||\to\infty,$ when the distance $||s_1-s_2||$ between the locations $s_1,s_2$ is preserved bounded. By this condition, the variances $Var(X(s))$ are not necessary uniformly bounded on $\mathbb{R}^d.$  Moreover, as we will see later, this condition allows to consider random fields with long-range dependence and their nonlinear transformations.

We are interested in the asymptotic behaviour of

$$\xi(\mu) = \frac{1}{\mu^{d}} \int\limits_{\Delta(\mu)}X(s)ds, \ \mbox{when} \ \mu\to\infty,$$ where $\Delta(\mu)$ is a homothetic transformation with the coefficient $\mu$ of a set $\Delta\in\mathbb{R}^d$ which plays a role of an observation window in statistical applications. For the case of identically distributed $X(s)$ the statistic $\xi(\mu)$ is a classical estimator of the mean. This paper studies the case of correlated observations and conditions on the behaviour of the covariance function $B(s_1,s_2)$ that guarantee the convergence $\xi(\mu)\xrightarrow{a.s.}0, \mu\to\infty.$

In \cite{lyons1988strong} SLLN was obtained for sequences of weakly dependent random variables $\{X(n), \ n \geq 1\}$ such that $Var(X(n)) = O(1).$ Other results on the SLLN can be found in \cite{moricz1977strong},  \cite{moricz1985slln}, \cite{serfling1970convergence}.

The multidimensional versions of the results were obtained in \cite{dokl1975criteria} and \cite{parker2019strong}. \cite{parker2019strong} proved the SLLN for 2-dimensional arrays  $\{X(n,m), \ n \geq 1, \ m \geq 1\}$ of independent random elements with values in Banach spaces.  In \cite{dokl1975criteria}, the multidimensional SLLN for stationary random processes and homogeneous random fields was established. Necessary and sufficient conditions for the SLLN were found.

Later \citet*{hu2005golden}  proved the SLLN for sequences of random variables $\{X(n), \ n \geq 1\}$ with less restrictive conditions on moments and dependencies between observations comparing with the results in \cite{lyons1988strong}. More precisely, they studied the case of
\[Var(X(n))<H(n), \ n\geq 1,$$ $$\sum_{n=1}^\infty\frac{H(n^\varphi)}{n^2}<\infty,\] 
where $\varphi$ is a golden ratio, and used the dependency assumption 
\[\sup_{n}|cov(X(n),X(n+m))|\leq\rho(m),\]  
\[\sum_{m=1}^{\infty}\frac{\rho(m)}{m^{\varphi-1}}<\infty.\] 
From these conditions, one may see that there are sequences of random variables possessing long-range dependence that satisfy SLLN. 

The results presented in \citet*{hu2005golden} were extended by the same authors to the case of more general normalization of partial sums in \citet*{MR2458009}. These results found numerous applications. We provide a few examples that illustrate areas that employed such results:\\
- in \citet*{MR3849560},  the results from \citet*{hu2005golden} were applied to an investigation of wireless massive multiple-input multiple-output system entails a large number of base station antennas serving a much smaller number of users, with large gains in spectral efficiency and energy efficiency; \\
- in \citet*{li2017efficient}, to data lying in a high dimensional ambient space that commonly thought to have a much lower intrinsic dimension; \\
- in \cite{MR3270987}, to a study of a Bayesian multichannel change-point detection problem in a general setting; \\
- in \cite{shu2019estimation}, to the estimation of large covariance and precision matrices from high-dimensional sub-Gaussian or heavier-tailed observations with slowly decaying temporal dependence; \\
- in \citet*{MR3959671}, to show that stochastic programming provides a framework to design hierarchical model predictive control schemes for periodic systems; \\
- in \citet*{pumi2019dynamic}, to a class of dynamic models for time series taking values on the unit interval;\\
- in \citet*{MR3997647}, to a recurring theme in modern statistics that is dealing with high-dimensional data whose main feature is a large number of variables but a small sample size;\\
- in \cite{vega2013stochastic}, to a construction of adaptive algorithms that leads to the so called Stochastic Gradient algorithms; \\
- in \citet*{hojjatinia2020identification}, to introduce novel methodologies for the identification of coefficients of switching autoregressive moving average with exogenous input systems and switched autoregressive exogenous linear models.

The main novelties of the obtained in this paper results are in investigating
\begin{itemize}
\item[--]  case of random fields with multidimensional observation windows;
\item[--] case of functional data and their integral functionals;
\item[--] weaker conditions on variances and dependencies then in the above publications;
\item[--] weakly and strongly dependent random fields.
\end{itemize}

Specifically, the conditions on the variances and covariance functions are relaxed. The variances are bounded by the functions $H(s)$ which can increase faster than in \citet*{hu2005golden}, as $||s||\to\infty.$ The covariance functions can unboundedly increase, when the locations are getting far away from the origin, but a distance between them is bounded. We show that wide classes of long-range dependent random fields and their nonlinear transformations satisfy the SLLN. Moreover, observation windows $\Delta$ can be taken from a wide class of sets.

This paper is organized as follows. Section~\ref{sec2} provides required definitions and notations. The main results of this article are proved in Section~\ref{sec3}. Numerical studies confirming the theoretical findings are given in Section~\ref{sec4}. Conclusions and some open problems are presented in Section~\ref{sec5}.

\section{Definitions and notations}\label{sec2}
The main definitions and notations about random fields are given in this section.

In what follows we use the symbol $C$ to denote constants which are not important for our discussion. Moreover, the same symbol $C$ may be used for different constants appealing in the same proof. Let $\overset{d}{=}$ denote the equality of finite-dimensional distributions.

For $\nu>-\frac{1}{2}$, we denote by $$J_{\nu}(z) = \sum_{m = 0}^{\infty}\frac{(-1)^m(z/2)^{2m+\nu}}{m!\Gamma(m+\nu+1)}, \ z\geq 0,$$ the Bessel's function of the first kind of order $\nu$, where $z \geq 0.$

\begin{definition} A random field $X(s), \ s \in \mathbb{R}^n ,$ is called strictly homogeneous, if finite-dimensional distributions of $Y(s)$ are invariant with respect to the group of motion transformations
$$P(X(s_1)<a_1, ...,X(s_k)<a_k) = P(X(s_1+h)<a_1,...,X(s_k+h)<a_k)$$  for all $h \in \mathbb{R}^n,$ where $s_1,..,s_k\in\mathbb{R}^n$ and $a_1,..,a_k\in\mathbb{R}.$
\end{definition}

\begin{definition} A random field $X(s), \ s \in \mathbb{R}^n,$ is called isotropic, if its finite-dimensional distributions are invariant with respect to rotation transformations $$P(X(s_1)<a_1,...,X(s_k)<a_k) =\\ P(X(As_1)<a_1,...,X(As_k)<a_k)$$ for all rotation transformations, i.e. orthogonal matrices $A$ with the absolute value of determinant of $A$ equals to 1,  $s_1,..,s_k\in\mathbb{R}^n$ and $a_1,..,a_k\in\mathbb{R}^d.$
\end{definition}

Let $X(s)$ be a constant mean random field defined on $\mathbb{R}^d, d \geq 1.$ Without loss of generality, let $EX(s)=0.$

 The function $B(r), \ r\geq0,$ is a correlation function of an isotropic random field if and only if there exists a measure $G$ on $\{ \mathbb{R}_{+}, \mathbb{B}_{+} \}$ such that $B(r)$ allows the following integral representation
$$B(r) = E\big( X(s_1)X(s_2) \big) = 2^\frac{n-2}{2}\Gamma\bigg(\frac{n}{2}\bigg) \int_{\mathbb{R}_{+}} J_{\frac{n-2}{2}}(ur)(ur)^{\frac{2-n}{2}}G(du), \ r=||s_1-s_2||, $$ where $||\cdot||$ denotes the Euclidean distance in $\mathbb{R}^d.$

\begin{definition} A measurable function $L:(0,\infty) \to (0,\infty)$ is called slowly varying at the infinity if for all $\lambda>0$
$$\lim_{t \to \infty} \frac{L(\lambda t )}{L(t)} = 1.$$
\end{definition}

\begin{definition} The function
$$H_m(u)=(-1)^me^{u^2/2}\frac{d^m}{du^m}e^{-\frac{u^2}{2}}$$ is a Hermite polynomial of order $m$. The first few Hermite polynomials are
$$H_0(u) = 1, \ H_1(u) = u, \ H_2(u) = u^2-1.$$
\end{definition}

It is known that the Hermite polynomials form a complete orthogonal system in the space $L_2(\mathbb{R}, \phi(u)du)$, i.e.
$$\int_{\mathbb{R}}H_{m_1}(u)H_{m_2}(u)\phi(u)du = \delta_{m1}^{m_2}m_1!,$$ where $\delta_{m_1}^{m_2}$ is a Kronecker delta function.

\begin{definition}
\label{long}
A random field $X(s),s\in\mathbb{R}^d,$ is called long-range dependent if its covariance function $B(s_1,s_2) = E(X(s_1)X(s_2))$ is not absolutely integrable for each $s_1\in \mathbb{R}^d,$ i.e.
$$\int\limits_{\mathbb{R}^d}|B(s_1,s_1+s_2)|ds_2 = +\infty.$$
\end{definition}

\section{Main results}\label{sec3}

In this section we introduce dependencies assumptions and prove the SLLN for random fields.

\begin{assumption}\label{assump1} The absolute value of the covariance function of $X(s)$ is bounded as
$$|B(s_1,s_2)| = |cov(X(s_1), X(s_2))|\leq C(1+||s_1||^\gamma+||s_2||^\gamma)\rho(||s_1-s_2||), \ \gamma \ge 0,$$ where $\rho(u), u\in\mathbb{R}_{+},$ is a such positive bounded function  that for some $\beta > 0$ it holds $\rho(u)\leq 1/u^{\beta}, \ u\geq 1.$

\end{assumption}

Then the variance of $X(s)$ is bounded by the function
$$H(s): = C(1+2||s||^\gamma)\rho(0),$$  which can unboundedly increase, when $||s||\to\infty.$

\begin{remark}
It follows from Assumption \ref{assump1} that for any fixed $s_1\in\mathbb{R}^d$ the covariance function is bounded by
$C(s_1)||s_2||^{\gamma-\beta}$ if $||s_2||\to+\infty$ and
$$\int_{\mathbb{R}^d}B(s_1,s_2)ds_2\leq C\left(1+C(s_1)\int_{||s_2||\geq 1}||s_2||^{\gamma-\beta}ds_2\right)$$ 
$$=C\left(1+C(s_1)\int_{1}^{\infty}r^{d+\gamma-\beta-1}dr\right).$$  Therefore, in the case $\beta-\gamma\geq d$  the random field $X(s)$ is weakly dependent and  in the case $\beta-\gamma <d$  the random field $X(s)$ can have a non-integrable covariance function and be long-range dependent.

\end{remark}

We consider the random variables
\begin{equation}
\label{funct}
\xi(\mu) = \frac{1}{\mu^{d}}\int_{\Delta(\mu)}X(s)ds,\ \mu>0,
\end{equation} 
where $\Delta(\mu)$ is a homothetic transformation with the parameter $\mu$ of a simply connected $d$\mbox{-}dimensional set $\Delta\subset\mathbb{R}^d,$ which is a compact set containing the origin and the Lebesgue measure $|\Delta| > 0$. The integral in \eqref{funct} is well-defined because of the measurability of $X(s)$, see Theorem 1.1.1. in \cite{ivanov2012statistical}.

We will use the notation $\rm{diam}(A)$ for the diameter of the set $A\subset\mathbb{R}^d,$ i.e. $\rm{diam}(A) = \sup_{x,y\in A}||x-y||. $

To establish conditions for $\xi(\mu)\xrightarrow{a.s.}0,$ as $\mu\to\infty,$ we use the classical method of subsequences. By this method, the existance of the increasing subsequence $\{\mu_n, n\geq 1\}$ such that $\xi(\mu_n)\xrightarrow{a.s.}0,$ as $n\to\infty,$  and the convergence of the deviations $\sup_{\mu\in[\mu_n,\mu_{n+1})}|\xi(\mu)-\xi(\mu_n)|\xrightarrow{a.s.}0,$ as $n\to\infty,$ are enough for the convergence $\xi(\mu)\xrightarrow{a.s.}0,$ as $\mu\to\infty.$

\begin{lemma}
\label{lemma1}
 Let Assumption~\ref{assump1} be satisfied and there exist an increasing sequence of positive numbers $\{ \mu_n, n\geq 1\}$ such one of the following conditions holds
\begin{itemize}
\item[(i)]
for $\beta < d$
\begin{equation}
\label{cond1}
\sum_{n=1}^{\infty}\frac{1}{\mu_n^{\beta-\gamma}}<+\infty ,
\end{equation}  
\item[(ii)]  for  $\beta \geq d$
\end{itemize}
\begin{equation}
\label{cond2}
\sum_{n=1}^{\infty}\frac{1}{\mu_n^{d-\gamma}}<+\infty.
\end{equation}  Then, the sequence of random variables $\xi(\mu_n)\xrightarrow{a.s.}0,$ as $n\to\infty.$
\end{lemma}

\begin{proof}
Note that as $E(X(s))=0,$ then the variances of $\xi(\mu)$ can be estimated as
$$Var(\xi(\mu)) =\frac{1}{\mu^{2d}} \int_{\Delta{(\mu)}}\int_{\Delta{(\mu)}} cov(X(s_1),X(s_2))ds_1ds_2$$
$$\leq\frac{C}{\mu^{2d}} \int_{\Delta{(\mu)}}\int_{\Delta{(\mu)}}(1+||s_1||^{\gamma}+||s_2||^{\gamma})\rho(||s_1-s_2||)ds_1ds_2$$
$$\leq \frac{C}{\mu^{2d}} \int_{\Delta{(\mu)}}\int_{\Delta{(\mu)}}(1+||s_1||^{\gamma})\rho(||s_1-s_2||)ds_1ds_2+\frac{C}{\mu^{2d}} \int_{\Delta{(\mu)}}\int_{\Delta{(\mu)}}(1+||s_2||^{\gamma})\rho(||s_1-s_2||)ds_1ds_2$$
$$= \frac{2C}{\mu^{2d}} \int_{\Delta{(\mu)}}\int_{\Delta{(\mu)}}(1+||s_1||^{\gamma})\rho(||s_1-s_2||)ds_1ds_2.$$

After the change of the variables $\tilde{s}_1 = s_1, \ \tilde{s}_2 = s_1-s_2$ one gets
$$Var(\xi(\mu)) = \frac{2C}{\mu^{2d}} \int\limits_{\Delta{(\mu)}}\int\limits_{\Delta{(\mu)}-\Delta{(\mu)}}(1+||\tilde{s}_1||^{\gamma})\rho(||\tilde{s}_2||)d\tilde{s}_1\tilde{ds}_2,$$ where $\Delta{(\mu)}-\Delta{(\mu)}: = \{s-\tilde{s}:s,\tilde{s} \in \Delta(\mu) \}$ denotes the Minkowski difference of sets.

The set $\Delta{(\mu)}$ is bounded, so there exists a centered ball $B(\mu \cdot \rm{diam}(\Delta)) = \{ s\in \mathbb{R}^n:||s|| \leq \mu \cdot \rm{diam}(\Delta)\}$ such that $\Delta{(\mu)}-\Delta{(\mu)} \subset B(\mu\cdot \rm{diam}(\Delta)).$ Then, by converting the integrals to the spherical coordinates, it follows from Assumption \ref{assump1} that
$$Var(\xi(\mu)) \leq  \frac{C}{\mu^{2d}} \int\limits_{\Delta{(\mu)}}(1+||\tilde{s_1}||^{\gamma})d\tilde{s_1}\int\limits_{B(\mu\cdot \rm{diam}(\Delta))}\rho(||\tilde{s_2}||)\tilde{ds_2}$$
$$\leq \frac{C}{\mu^{2d}}\int\limits_{0}^{\mu\cdot \rm{diam}{\Delta(1)}}\tilde{r}_1^{d-1}(1+\tilde{r}_1^\gamma)d\tilde{r}_1\int\limits_{0}^{\mu \cdot \rm{diam}(\Delta)}\tilde{r}_2^{d-1}\rho(\tilde{r}_2)d\tilde{r}_2$$
$$\leq \frac{C}{\mu^{2d}}\big(\mu^d + \mu^{d+\gamma} \big)\bigg(C +  \int\limits_{1}^{\mu \cdot \rm{diam}(\Delta)}\tilde{r}_2^{d-1}\rho(\tilde{r}_2)d\tilde{r}_2\bigg)\leq \frac{C}{\mu^{2d}}\big(\mu^d + \mu^{d+\gamma} \big)\big(C +  \mu^{d-\beta}\big).$$

Now the Borel–Cantelli lemma is used to find conditions for $\xi(\mu_n)\xrightarrow{a.s.}0.$

By  Chebyshev's inequality one gets
$$\sum_{n=1}^{\infty}P(|\xi(\mu_n)| > \varepsilon) \leq \frac{1}{\varepsilon^2}\sum_{n=1}^{\infty}Var(\xi(\mu_n))\leq C\sum_{n=0}^{\infty}\frac{\big(\mu_n^d + \mu_n^{d+\gamma} \big)\big(C +  \mu_n^{d-\beta}\big)}{\mu_n^{2d}}.$$

Let $\beta < d,$ then
$$\frac{\big(\mu_n^d + \mu_n^{d+\gamma} \big)\big(C+  \mu_n^{d-\beta}\big)}{\mu_n^{2d}}\leq \frac{C}{\mu_n^{\beta-\gamma}}, \ \ \ \ \ \ n\to\infty.$$

 Thus, for $\beta < d$ the sequence $\xi(\mu_n)\xrightarrow{a.s.}0,$ when $n\to\infty,$ if
$\sum_{n=1}^{\infty}{\mu_n^{\gamma-\beta}}<+\infty.$

Let $\beta \geq d.$ Then
$$\frac{\big(\mu_n^d + \mu_n^{d+\gamma} \big)\big(C +  \mu_n^{d-\beta}\big)}{\mu_n^{2d}}\leq  \frac{C}{\mu_n^{d-\gamma}}.$$

Thus, for $\beta\geq d$ the sequence $\xi(\mu_n)\xrightarrow{a.s.}0,$ when $n\to\infty,$ if
$$\sum_{n=1}^{\infty}\frac{1}{\mu_n^{d-\gamma}}<+\infty.\vspace{-6mm}$$
\end{proof}

\begin{lemma}
\label{lemma2}
Let Assumption \ref{assump1} be satisfied. If there exists an increasing sequence of positive numbers $\{ \mu_n, n\geq 1\}$ such that
\begin{equation}
\label{cond3}
\sum_{n=1}^{\infty}  \frac{(\mu_{n+1}^d - \mu_{n}^d)^2\mu_n^{\gamma}}{\mu_n^{2d}} < + \infty
\end{equation} and
\begin{equation}
\label{cond4}
\sum_{n=1}^{\infty}\bigg(\frac{1}{\mu_{n+1}^d}-\frac{1}{\mu_n^d}\bigg)^2 \mu_{n+1}^{2d+{\gamma}}<+\infty,
\end{equation} 
then the sequence of random variables $\eta(\mu_n) := \sup_{\mu \in [\mu_n, \mu_{n+1})}\big|\xi(\mu)-\xi(\mu_n)\big| \xrightarrow{a.s.}0,\ n\to\infty.$
\end{lemma}

\begin{proof}
It is enough to show that the random variables  $\eta(\mu_n)$  are bounded by a sequence of random variables converging a.s. to $0$,  when $n \to \infty$.

The random variables $\eta(\mu_n)$ allow the following  estimation from above
$$\sup_{\mu\in[\mu_n,\mu_{n+1})}\big|\xi(\mu) - \xi(\mu_n)\big| = \sup_{\mu\in[\mu_n,\mu_{n+1})}\bigg|\frac{1}{\mu^d}\int_{\Delta(\mu)}X(s)ds - \frac{1}{\mu_n^d}\int_{\Delta(\mu_n)}X(s)ds\bigg| $$
$$\leq  \sup_{\mu\in[\mu_n,\mu_{n+1})}\bigg|\frac{1}{\mu_n^d}\int \limits_{\Delta(\mu)\setminus \Delta(\mu_n)}X(s)ds\bigg|+\sup_{\mu\in[\mu_n,\mu_{n+1})}\bigg|\bigg(\frac{1}{\mu^d}-\frac{1}{\mu_n^d}\bigg)\int_{\Delta(\mu)}X(s)ds \bigg|.$$

The above supremums can be estimated as
$$ \sup_{\mu\in[\mu_n,\mu_{n+1})}\bigg|\frac{1}{\mu_n^d}\int \limits_{\Delta(\mu)\setminus \Delta(\mu_n)}X(s)ds\bigg| \leq  \frac{1}{\mu_n^d}\int \limits_{\Delta(\mu_{n+1})\setminus \Delta(\mu_n)}|X(s)|ds := I_{\mu_n}^{(1)},$$
$$\sup_{\mu\in[\mu_n,\mu_{n+1})}\bigg|\bigg(\frac{1}{\mu^d}-\frac{1}{\mu_n^d}\bigg)\int\limits_{\Delta(\mu)}X(s)ds \bigg|          \leq            \bigg(\frac{1}{\mu_n^d}-\frac{1}{\mu_{n+1}^d}\bigg)\int\limits_{\Delta(\mu_{n+1})}|X(s)|ds := I_{\mu_n}^{(2)}.$$

The next step is to find conditions that guarantee that $I_{\mu_n}^{(1)}\xrightarrow{a.s.}0$ and $I_{\mu_n}^{(2)}\xrightarrow{a.s.}0,$ when $n\to\infty.$

Using Markov's inequality, the following series  can be estimated as
$$\sum_{n=1}^{\infty} P(I_{\mu_n}^{(1)}>\varepsilon)  = \sum_{n=1}^{\infty}P\bigg( \frac{1}{\mu_n^d}\int \limits_{\Delta(\mu_{n+1}) \setminus \Delta(\mu_n)}|X(s)|ds > \varepsilon\bigg)$$
$$\leq\frac{1}{\varepsilon^2}\sum_{n=1}^{\infty} \frac{1}{\mu_n^{2d}}E\bigg( \int \limits_{\Delta(\mu_{n+1}) \setminus \Delta(\mu_n)}|X(s)|ds\bigg)^2  = \frac{1}{\varepsilon^2}\sum_{n=1}^{\infty}\frac{1}{\mu_n^{2d}} \ \ \iint \limits_{\big(\Delta(\mu_{n+1}) \setminus \Delta(\mu_n)\big)^2}E|X(s_1)||X(s_2)|ds_1ds_2,$$ 
where $\iint\limits_{A^2}$ denotes the double integral $\int_A\int_A.$

By using H\"older's inequality for $p=q=2$ one gets
$$\sum_{n=1}^{\infty} P(I_{\mu_n}^{(1)}>\varepsilon) \leq \frac{1}{\varepsilon^{2}}\sum_{n=1}^{\infty} \frac{1}{\mu_n^{2d}} \ \ \iint \limits_{\big(\Delta(\mu_{n+1}) \setminus \Delta(\mu_n)\big)^2}\bigg(EX^2(s_1)EX^2(s_2) \bigg)^{1/2} ds_1ds_2 $$
$$=
\frac{1}{\varepsilon^2}\sum_{n=1}^{\infty} \frac{1}{\mu_n^{2d}}  \ \ \bigg(\int \limits_{\Delta(\mu_{n+1}) \setminus \Delta(\mu_n)}\big(EX^2(s)\big)^{1/2} ds\bigg)^{2} $$
$$\leq \frac{|\Delta(1)|^2}{\varepsilon^2}\sum_{n=1}^{\infty} \frac{(\mu_{n+1}^{d} - \mu_{n}^{d})^2\sup_{s\in \Delta(\mu_{n+1}) \setminus \Delta(\mu_n)}H(s)}{\mu_n^{2d}}  \leq  \frac{C}{\varepsilon^2}\sum_{n=1}^{\infty} \frac{(\mu_{n+1}^{d} - \mu_{n}^{d})^2(1+C\mu_{n+1}^{\gamma})}{\mu_n^{2d}}.$$

Thus, by the Borel-Cantelli Lemma, the sequence $I_{\mu_n}^{(1)}\xrightarrow{a.s.}0$, when $n\to\infty,$ if
$$\sum_{n=1}^{\infty} \frac{(\mu_{n+1}^{d} - \mu_{n}^{d})^2\mu_{n+1}^{\gamma}}{\mu_n^{2d}}< + \infty.$$

Now we derive conditions for the convergence $I_{\mu_n}^{(2)}\xrightarrow{a.s.}0,$ when $n\to\infty.$

Using Markov's and H\"older's inequalities one obtains
$$\sum_{n=1}^{\infty}P(I_{\mu_n}^{(2)}>\varepsilon) =  \sum_{n=1}^{\infty}P\bigg(\bigg(\frac{1}{\mu_{n}^{d}}-\frac{1}{\mu_{n+1}^{d}}\bigg)\int_{\Delta(\mu_{n+1})}|X(s)|ds > \varepsilon\bigg)\leq \frac{1}{\varepsilon^2}\sum_{n=1}^{\infty}\bigg(\frac{1}{\mu_{n+1}^{d}}-\frac{1}{\mu_{n}^{d}}\bigg)^2$$
 $$\times E\bigg(\int_{\Delta(\mu_{n+1})}|X(s)|ds \bigg)^2=\frac{1}{\varepsilon^2}\sum_{n=1}^{\infty}\bigg(\frac{1}{\mu_{n+1}^{d}}-\frac{1}{\mu_{n}^{d}}\bigg)^2\iint\limits_{\big(\Delta(\mu_{n+1})\big)^2}E|X(s_1)||X(s_2)|ds_1ds_2$$
$$\leq \frac{1}{\varepsilon^2}\sum_{n=1}^{\infty}\bigg(\frac{1}{\mu_{n+1}^{d}}-\frac{1}{\mu_{n}^{d}}\bigg)^2\bigg(\int\limits_{\Delta(\mu_{n+1})}\bigg(EX^2(s)\bigg)^{1/2}ds\bigg)^2 \leq \frac{C}{\varepsilon^2}\sum_{n=1}^{\infty}\bigg(\frac{1}{\mu_{n+1}^{d}}-\frac{1}{\mu_{n}^{d}}\bigg)^2 (1+C\mu_{n+1}^{{\gamma}})\mu_{n+1}^{2d}.$$

Hence, by the Borel-Cantelli Lemma, the sequence $I_{\mu_n}^{(2)}\xrightarrow{a.s.}0$, when $n\to\infty,$ if
$$\sum_{n=1}^{\infty}\bigg(\frac{1}{\mu_{n+1}^d}-\frac{1}{\mu_n^d}\bigg)^2 \mu_{n+1}^{{\gamma}+2d}<+\infty.$$

From the positiveness of $\eta(\mu_n) = \sup_{\mu \in [\mu_n, \mu_{n+1})}\big|\xi(\mu)-\xi(\mu_n)\big|$  and the boundedness $\eta(\mu_n) \leq I_{\mu_n}^{(1)}+I_{\mu_n}^{(2)}$ such that $I_{\mu_n}^{(1)}\xrightarrow{a.s.}0$ and $I_{\mu_n}^{(2)}\xrightarrow{a.s.}0$ it follows the convergences $$\sup_{\mu \in [\mu_n,\mu_{n+1})}\big|\xi(\mu) - \xi(\mu_n)\big|\xrightarrow{a.s.}0, \ n\to\infty.\vspace{-5mm}$$
\end{proof}

\begin{theorem}
\label{theorem1}
If for $\beta>0$ and $\gamma>0$ there exists a sequence $\{\mu_n, n\geq 1\}$ satisfying the assumptions in Lemmas \ref{lemma1} and \ref{lemma2}, then $\xi(\mu)\xrightarrow{a.s.}0,$ as $\mu\to\infty.$

\end{theorem}

\begin{proof}

As $\{ \mu_n, n \geq 1 \}\subset \mathbb{R}_{+}$ is an increasing sequence and $\mu_n \to \infty, \ n\to\infty,$ then for each $\mu\in\mathbb{R}_{+}$ there exists $\mu_n$ such that $\mu \in [\mu_n, \mu_{n+1}).$ It follows from  $\xi(\mu) = \xi(\mu_n) + \xi(\mu) - \xi(\mu_n)$ that
$$\xi(\mu_n) - \sup_{\mu \in [\mu_n,\mu_{n+1})}\bigg|\xi(\mu) - \xi(\mu_n)\bigg| \leq\xi(\mu) \leq \xi(\mu_n) + \sup_{\mu \in [\mu_n,\mu_{n+1})}\bigg|\xi(\mu) - \xi(\mu_n)\bigg|.$$

Thus, by Lemma \ref{lemma1} and by Lemma \ref{lemma2} we get $\xi(\mu)\xrightarrow{a.s.}0,$ as $\mu\to\infty.$
\end{proof}

\begin{theorem}
\label{theorem2}
Let Assumption \ref{assump1} be satisfied. The SLLN holds true, if one of the following conditions is satisfied
\begin{itemize}
\item[(i)] $\beta\in(0,d)$ and $2\gamma<\beta,$
\item[(ii)] $\beta \geq d$ and $2\gamma<d.$
\end{itemize}

\end{theorem}
\begin{proof}

Let $\mu_n = n^\alpha.$ Consider the first case and check the conditions of Lemmas \ref{lemma1} and \ref{lemma2}.

The condition \eqref{cond1} becomes
$$\sum_{n=1}^{\infty}\frac{1}{n^{\alpha(\beta-\gamma)}}<\infty,$$ and is satisfied if $\alpha > \frac{1}{\beta-\gamma}.$

The conditions \eqref{cond3} and \eqref{cond4} hold if $\alpha\gamma<1.$ Indeed, from the asymptotic behavior of the terms in \eqref{cond3} and  \eqref{cond4}
$$\frac{(\mu_{n+1}^d - \mu_{n}^d)^2\mu_n^{\gamma}}{\mu_n^{2d}} = \frac{((n+1)^{\alpha d} - n^{\alpha d})^2n^{\alpha\gamma}}{n^{2\alpha d}}\sim \frac{\big((n+1)^{\alpha d-1}\big)^2n^{\alpha\gamma}}{n^{2\alpha d}} = \frac{1}{n^{2 - \alpha\gamma}}$$ and
$$\bigg(\frac{1}{\mu_{n+1}^d}-\frac{1}{\mu_n^d}\bigg)^2 \mu_{n+1}^{\gamma+2d} = \bigg(\frac{1}{(n+1)^{\alpha d}}-\frac{1}{n^{\alpha d}}\bigg)^2 (n+1)^{\alpha\gamma+2\alpha d}$$
$$\sim \frac{(n+1)^{\alpha\gamma+2\alpha d}}{(n+1)^{2\alpha d+2}} = \frac{1}{(n+1)^{2-\alpha\gamma}}.$$

Thus, the conditions of Lemmas \ref{lemma1} and \ref{lemma2} are satisfied if $0 < \beta < d,\ \alpha > \frac{1}{\beta-\gamma}$ and $\alpha\gamma < 1.$ Then, it follows from $\frac{1}{\beta-\gamma}<\alpha<\frac{1}{\gamma}$ that the required $\alpha$ exists if $2\gamma<\beta.$

Using the same approach, we derive that for $\beta\geq d$ the conditions of Lemmas \ref{lemma1} and \ref{lemma2} are satisfied if $\beta \geq d$ and $\frac{1}{d-\gamma}<\alpha<\frac{1}{\gamma}.$ Thus, the required $\alpha$ exists if $2\gamma < d.$
\end{proof}

\begin{remark}
\label{remark1}
The upper bound in Assumption \ref{assump1} can be replaced by another one that guarantees for $||s_1||,||s_2||\to+\infty$ the covariance $cov(X(s_1), X(s_2))$ sufficiently fast decays to zero for $s_1$ and $s_2$ which are getting further away from each other, and it  increases not too fast for $s_1$ and $s_2$ that are close. For instance, one can use the conditions
$$|B(s_1,s_2)|\leq  C(1+||s_1||+||s_2||)^\gamma\rho(||s_1-s_2||)$$ or
$$|B(s_1,s_2)|\leq  C(1+||s_1||^\gamma)(1+||s_2||^\gamma)\rho(||s_1-s_2||).$$
\end{remark}

\begin{remark}
\label{remark2}
Homogeneous isotropic random fields satisfy Assumption \ref{assump1} with $\gamma=0,$ if their covariance functions have hyberbolic bounded decays of order $\beta.$
\end{remark}

\section{Non-stationary example}~\label{sec_ex}
As SLLN holds for homogeneous isotropic random fields with hyperbolically bounded covariance functions, it would be interesting to provide a simple example of non-homogeneous and non-isotropic random field for which the result holds true.

\begin{example}
Let $X(s) = g(s)H_k(Z(s)),\ s \in \mathbb{R}^d, \ d\geq1,$ where $g(\cdot)$ is a deterministic function, $H_k(\cdot), k\in \mathbb{N},$ is the Hermite polynomial of degree $k$ and $Z(\cdot)$ is a homogeneous isotropic  Gaussian random field with $EZ(s)=0$ and the covariance function $B(s),$ such that $B(0)=1$ and
$$B_Z(s)=E\big(Z(s)Z(0)\big)=\frac{L(||s||)}{||s||^{\beta_0}}, \ \beta_0 >0,$$ where $L(s)$ is a slowly varying function.

By properties of the Hermite polynomials of Gaussian random variables, see, for example, $(2.1.8)$ in \cite{ivanov2012statistical}
$$EX(s) = g(s)EH_k(Z(s))=0,$$
$$B(s_1,s_2)= g(s_1)g(s_2)E\big(H_k(Z(s_1))H_k(Z(s_2)) \big)= g(s_1)g(s_2)k!B^k(||s_1-s_2||).$$

By properties of slowly varying functions, see Proposition $1.3.6(v)$ in \citet*{bingham1989regular}, for any $\beta > k\beta_0$ there is a constant C such that
$$B_Z^k(||s||)\leq\frac{C}{||s||^\beta},\ ||s||\geq1.$$

Thus, if
\begin{equation}
\label{eq1}
|g(s_1)g(s_2)|\leq C(1+||s_1||^\gamma+||s_2||^\gamma)
\end{equation} and $k\beta_0<\beta,$ then Assumption \ref{assump1} holds true and by Theorem \ref{theorem2}
$$\frac{1}{\mu^{d}}\int_{\Delta(\mu)}g(s)H_k(Z(s))ds\xrightarrow{a.s.}0, \ \mu\to+\infty.$$

First, note that it follows from
\begin{equation}
\label{eq2}
|g(s)|\leq C(1+||s||^{\gamma_0}), \ \gamma_0 >0,
\end{equation} that
$$|g(s_1)g(s_2)|\leq C(1+||s_1||^{2\gamma_0}+||s_2||^{2\gamma_0}).$$ Thus, if \eqref{eq2} holds, then \eqref{eq1} is true with $\gamma = 2 \gamma_0.$

\end{example}
Some example of functions $g(\cdot)$ satisfying \eqref{eq1} are
\begin{itemize}
\item[(i)] $g(s) \equiv C > 0.$ This case corresponds to the classical equally-weighted average functionals of homogeneous isotropic process or field;

\item[(ii)] $g(s)=\prod_{i=1}^ds_i^{l_i},$ where $s=(s_1,..,s_d), \ l_i>0, \ i=1,..,d.$ Note that
$$|g(s)| = \prod_{i=1}^d|s_i|^{l_i}\leq 1 + ||s||^{2\sum_{i=1}^dl_i}$$ and \eqref{eq1} is satisfied with $\gamma_0=2\sum_{i=1}^dl_i;$

\item[(iii)] $g(s)=\prod_{i=1}^ds_i\ln (q_i+|s_i|),$ where $s=(s_1,...,s_d)$ and $q_i>1,\ i=1,..,d.$
\end{itemize}
By using the logarithm inequality $\ln(x)\leq x-1,$ one obtains that
$$|g(s)|\leq C \prod_{i=1}^d|s_i|+\prod_{i=1}^d|s_i|^2$$ and the upper bound follows from the estimate in (ii) and \eqref{eq2}.

The weight functions in (ii) and (iii) are often used in non-linear regression and $M$ estimators applications.

It follows from results in \cite{alodat2020, ivanov2012statistical} that for the field $X(s)$ in the examples above one can obtain not only SLLN, but also limit theorems about the convergence of distributions. Namely, the following result holds true.

\begin{theorem}\cite{alodat2020} \label{tareq} Let a function $g(s), s \in \mathbb{R}^d,$ satisfy the condition $\mu^{2d-\beta_0k}g^2(\mu\cdot1_d)L^k(\mu)\to\infty,$ when $\mu\to\infty,$ and there exists a function $g^{*}(\cdot)$ such that
$$\lim_{\mu\to\infty}\bigg|\frac{g(\mu s)}{g(\mu 1_d)} - g^{*}(s)\bigg|\to 0$$ uniformly for $s\in\Delta(1+\varepsilon)$ for some $\varepsilon>0,$ 
$$\iint\limits_{\big(\Delta(1+\varepsilon) \big)^2} \frac{|g^{*}(s_1)g^{*}(s_2)|}{||s_1-s_2||^{\beta_0\kappa}}ds_1 ds_2 < +\infty, \quad \int_{\mathbb{R}^{dk}}\prod_{j=1}^k ||\lambda_j||^{\beta_0-d}|K_{\Delta}(\lambda_j, g^{*})|^2\prod_{j=1}^k d\lambda_j < +\infty,$$ and
$$\lim_{\mu\to\infty}\int\limits_{\mathbb{R}^{dk}}\bigg| \int\limits_{\Delta}e^{i(\lambda_1+..+\lambda_k,s)}\bigg(\frac{g(\mu ||s||)}{g(\mu \cdot 1_d||s||)}\prod_{j=1}^k\sqrt{\frac{L(\mu/||\lambda_j||)}{L(\mu)}} -g^{*}(s)\bigg)ds \bigg|^2\prod_{j=1}^k ||\lambda^{\beta_0-d}||\prod_{j=1}^kd\lambda_j = 0.$$  Then, for $\beta_0 \in \left(0,\min\left(\frac{d}{k}, \frac{d+1}{2}\right)\right)$ the random variables
$$\frac{1}{\mu^{d-\beta_0 k/2}g(\mu \cdot 1_d)L^{k/2}(\mu)c_1^{k/2}(d, \beta_0)}\int_{\Delta(\mu)}g(s)H_k(Z(s))ds$$ converge weakly to the random variable
$$\xi^{*}:=\int_{\mathbb{R}^{dk}}^{\prime} K_{\Delta}(\lambda_1+..+\lambda_k, g^{*})\frac{\prod_{j=1}^kW(d\lambda_j)}{\prod_{j=1}^k||\lambda_j||^{(d-\beta_0)/2}},$$ where $W(\cdot)$ is the complex Gaussian white noise random measure on $\mathbb{R}^d$, $\int_{\mathbb{R}^d}^{\prime}$ denotes the multiple Wiener-It\^{o} integral, where the diagonal hyperplanes  $\lambda_i = \pm \lambda_j,\ i,j = 1,..,k,\ i\neq j,$ are excluded from the domain of integration, $1_d=(1,..,1)\in\mathbb{R}^d,\ K_{\Delta}(\lambda, g^{*}) = \int_{\Delta}e^{i(\lambda, s)}g^{*}(s)ds,$ $c_1(d, \beta_0) = \Gamma((d-\beta)/2)/2^{\beta_0}\pi^{d/2}\Gamma(\beta_0/2).$
\end{theorem}
\begin{remark}
For the three functions $g(\cdot)$ introduced in the Example of it is easy to see that
\begin{itemize}
\item[(i)] $g^{*}(s) \equiv C,$
\item[(ii)] $g^{*}(s) = \prod_{i=1}^d  s_i^{l_i},$
\item[(iii)] $g^{*}(s)=\prod_{i=1}^d s_i.$
\end{itemize}
\end{remark}
\begin{remark}
For $\beta_0\in(0, \min(\frac{d}{k}, \frac{d+1}{2}))$ the random field $Z(s)$ is long-range dependent and the limit $\xi^{*}$ has  a non-Gaussian distribution if $k \geq 2.$
\end{remark}

\begin{remark}
If the random field $X(s)$ is weak-dependent, one can derive the Central Limit Theorem for the integral functionals of the form \eqref{funct}, see, for example, Theorems 1.7.1-1.7.3 in \cite{ivanov2012statistical}.
\end{remark}

\section{Numerical example}~\label{sec4}
In this section, we provide a numerical example confirming the obtained theoretical results. By simulations of random fields, we show that for the function $g(\cdot)$ satisfying (ii) in the Example in Section~\ref{sec_ex} the integral functional in \eqref{funct}  converges to $0$, as $\mu\to\infty.$ A reproducible version of the code in this paper is available in the folder "Research materials" from the  website~\url{https://sites.google.com/site/olenkoandriy/}.

We consider $d=2,$  the random variables in \eqref{funct} and the random field X(s)  given by  the formula
$$X(s) = \prod_{k=1}^2|s_k|^{\gamma}H_2(Z(s)),$$  where $H_2(x)=x^2-1$ is the Hermite polynomial of order $2$, $Z(s),s=(s_1,s_2)\in \mathbb{R}^2,$ is a homogeneous isotropic Gaussian random field with  the Cauchy type covariance function
$$B_Z(r) = \frac{1}{(1+r^2)^\beta},\ r\geq0.$$ The observation window $\Delta(1) = \square(1):= \{(s_1,s_2):|s_1|\leq 1, |s_2|\leq 1\}$ is a square.

For the simulations we used the values of  $\beta = 0.4$ and $\gamma = 0.1.$  Theorems  \ref{theorem2} and \ref{tareq} hold true for these values.

As the simulations of random fields can be done only on a discrete grid,  we used the dense grid of points $\{ (ih,jh): i,j = -N,-N+1,...,N-1,N \}, N\in\mathbb{N},$ where $h$ is a small fixed step. The integrals in \eqref{funct} were approximated by the  Riemann's sums
$$\int\limits_{\Delta(\mu)}X(s)ds\approx\sum\limits_{i=-N}^N\sum\limits_{j=-N}^NX(ih,jh)h^2=\sum\limits_{i=-N}^N\sum\limits_{j=-N}^N|ih|^{0.1}|jh|^{0.1}H_2(Z(ih,jh))h^2.$$

Then $300$ realizations of the random field $X(s)$ in the square region $\square(300)=\{(s_1,s_2): |s_i|\leq 300, i = 1,2 \}$ were generated.  A realization of the random fields $X(s)$ in the square $\square(300)$ on the 2D grid with the step $h = 0.25$ and the corresponding values of  $\xi(\mu)$ for $\mu = 10, 50, 100,...,300$  are given in Figure~\ref{fig1}. The Q-Q plot of the simulated values of $\xi(300)$ is shown in Figure \ref{fig2a}. As $\mu=300$ is sufficiently large the distribution is close to the asymptotic one. As expected, it is not Gaussian.

\begin{figure}[htb!]
\begin{subfigure}{0.44\textwidth}
  \centering
  \includegraphics[width=1\linewidth,height=6.2cm,trim=0 4mm 0 0,clip]{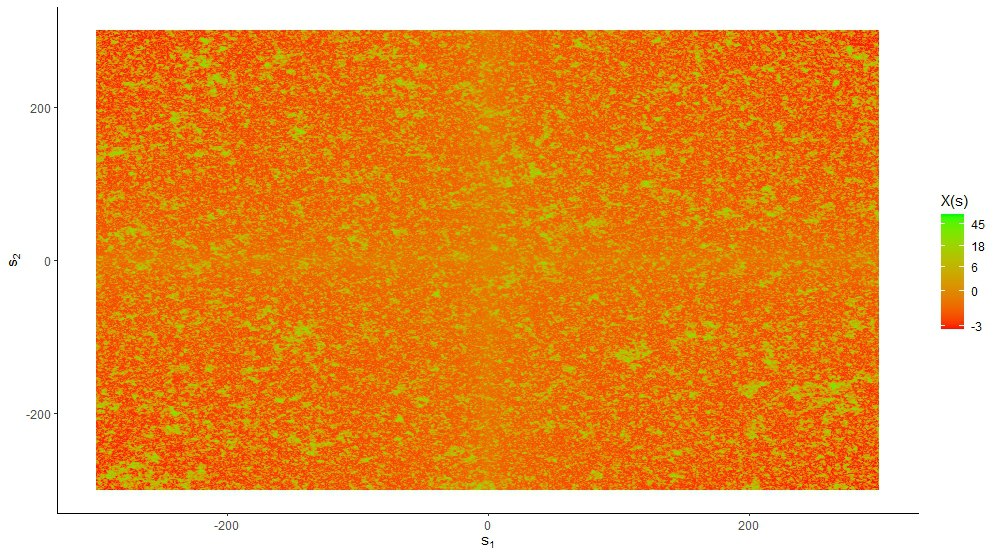}
  \caption{ Realization of $X(s)$}
  \label{fig1a}
\end{subfigure}%
\hspace{3mm}
\begin{subfigure}{0.45\textwidth}
  \centering
  \includegraphics[width=1\linewidth,height=6.4cm,trim=0 6mm 0 7mm,clip]{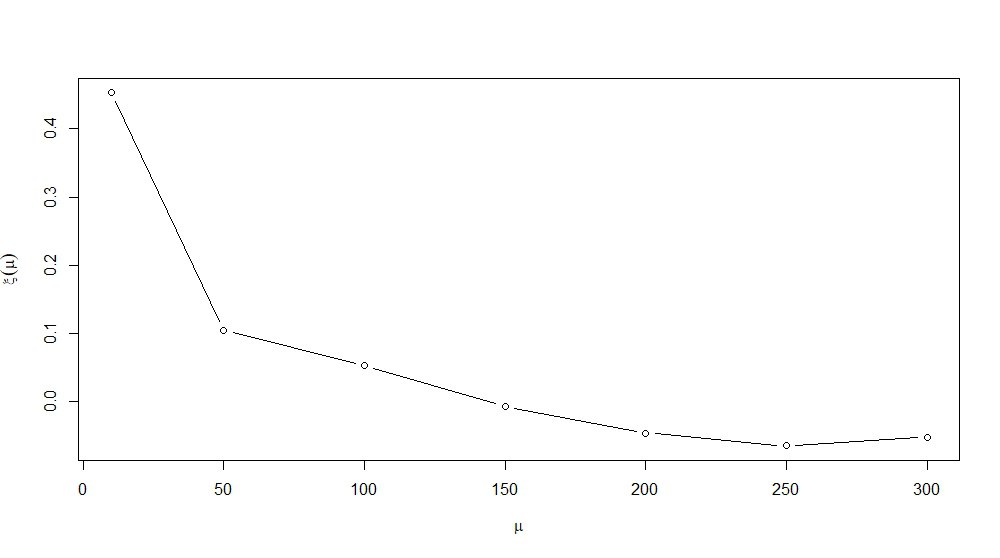}
  \caption{ Realization of $\xi(\mu)$}
  \label{fig1b}
\end{subfigure}
\caption{Realizations of the random field and its integral functional}
\label{fig1}
\end{figure}

Using the obtained realizations of  $X(s),$ the random variables $\xi(\mu)$ were computed for $\mu = 10, 50, 100, .., 300.$ The box plots of the simulated values of $\xi(\mu)$ are given in Figure~\ref{fig2b}. Table~\ref{rmse} shows the corresponding Root Mean Square Error (RMSE) of  $\xi(\mu)$ for different values of $\mu.$ Figure~\ref{fig2b} and the table confirm the convergence of $\xi(\mu)$ to zero when $\mu$ increases.

\begin{table}[hb]
\begin{center}
 \begin{tabular}{|c|c c c c c c c|}
 \hline
$\mu$ \rule[-2mm]{0pt}{7mm}& 10 & 50 & 100 & 150 & 200 & 250 & 300 \\
 \hline
RMSE\rule[-2mm]{0pt}{7mm} & 0.217 & 0.106 & 0.079 & 0.068 & 0.057 & 0.052 & 0.048 \\ [1ex]
 \hline
\end{tabular}
 \caption{RMSE of  $\xi(\mu)$.\label{rmse}}
 \end{center}
\end{table}

\begin{figure}[htb!]
\begin{subfigure}{0.45\textwidth}
  \centering
  \includegraphics[width=1\linewidth, height=6cm]{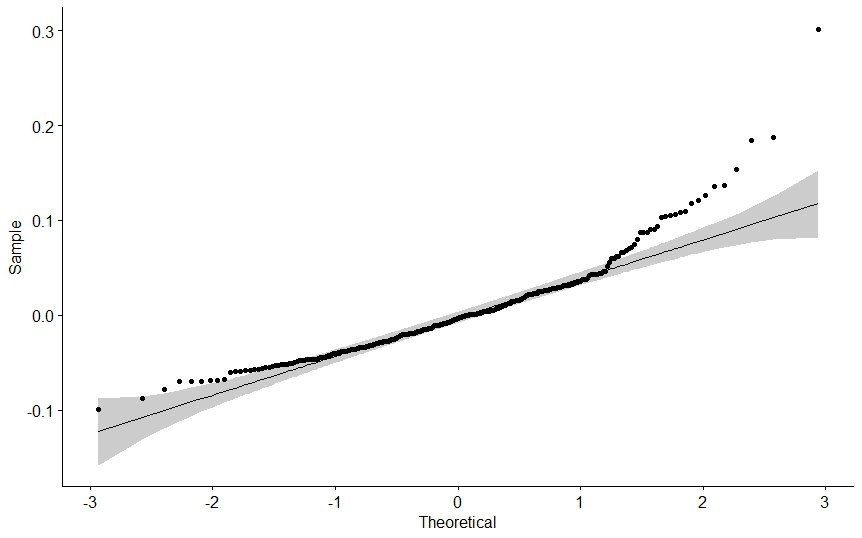}
  \caption{QQplot of  $\xi(300)$}
  \label{fig2a}
\end{subfigure}
\begin{subfigure}{0.45\textwidth}
  \centering
  \includegraphics[width=1\linewidth,height=6cm]{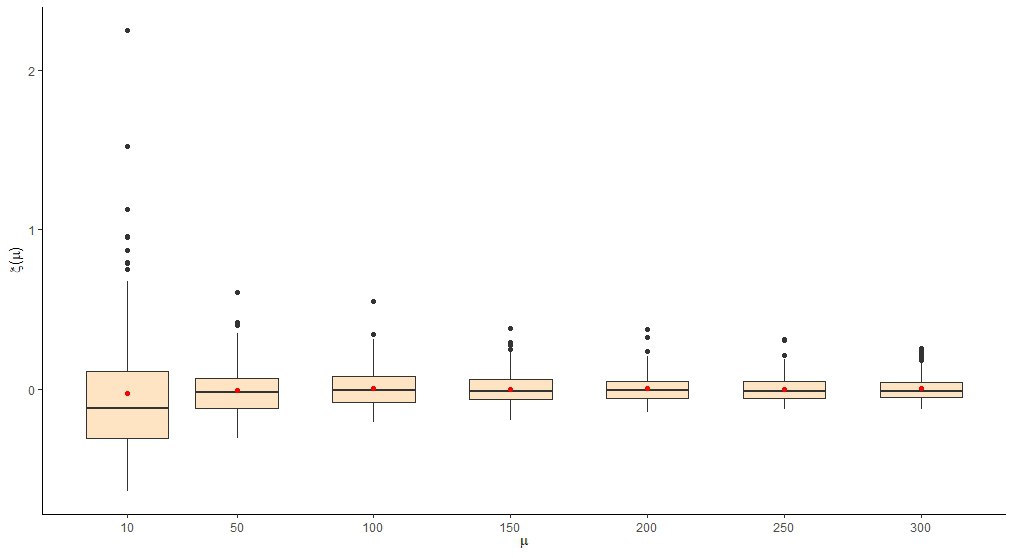}
  \caption{Boxplots of $\xi(\mu)$}
  \label{fig2b}
\end{subfigure}%
\caption{Empirical distributions of $\xi(\mu)$}
\label{fig2}
\end{figure}

\section{Conclusions and the future studies}~\label{sec5}
The SLLN for random fields with unboundedly increasing variances and covariance functions was obtained. The conditions of the obtained results allow to consider the case of  nonlinear transformations of long-range dependent random fields. The results were derived for a very general class of simply connected observation windows $\Delta.$

In the future studies, it would be interesting to obtain:

- Laws of Large Numbers with the complete convergence, see \citet*{hu2012complete}, for multidimensional functional data;

- Necessary and sufficient conditions for the SLLN for non-homogeneous and non-isotropic random fields \cite{dokl1975criteria};

- Rate of convergence in the SLLN, see \citet*{anh2019rate, Hu2020}.

\section*{Acknowledgements}
The authors would like to thank the anonymous reviewers for their suggestions that helped to improve the style of the paper.
\section*{Funding}
This research was supported under La Trobe University SEMS CaRE Grant "Asymptotic analysis for point and interval estimation in some statistical models". The research of the last listed author was partially funded by the subsidy allocated to Kazan Federal University for the state assignment in the sphere of scientific activities, project 1.13556.2019/ 13.1.
\bibliographystyle{chicago}
\bibliography{mybib_MN}
\end{document}